\documentclass[10pt]{article}
\usepackage{amsmath}
\usepackage{amssymb}
\usepackage{amsthm}
\usepackage{graphicx}
\usepackage{verbatim}
\usepackage{enumerate}
\usepackage{color}
\usepackage{mathrsfs}
\usepackage{longtable}
\usepackage{bera}

\numberwithin{equation}{section}

\theoremstyle{plain}
\newtheorem{thm}{Theorem}[section]

\newtheorem{lemma}[thm]{Lemma}
\newtheorem{proposition}[thm]{Proposition}
\newtheorem{corollary}[thm]{Corollary}

\theoremstyle{remark}

\newtheorem{remark}{Remark}

\allowdisplaybreaks

\def\CC{{\cal C}}

\def\Var{{\rm Var} \,}

\def\E{{\mathbf E}}
\def\P{{\mathbf P}}
\def\Cox{\hfill \Box}

\def\one{{\bf 1}}
\def\|{{\, | \, }}
\def\F{{\mathcal F}}
\def\T{{\mathcal T}}

\def\CC{{\mathcal C}}
\def\Zt{\widetilde{Z}}

\def\conn{\leftrightarrow}

\def\GW{{\tt GW}}
\def\Bin{\text{Bin}}
\def\one{\mathbf{1}}

\def\rtt{{\mathbf{0}}}
\def\TT{{\bf T}}

\reversemarginpar

\begin{document}

	\title{Critical Percolation and the Incipient Infinite Cluster on Galton-Watson Trees}
	\author{Marcus Michelen\thanks{University of Pennsylvania, Department of Mathematics, David Rittenhouse Lab. 209 South 33rd Street Philadelphia, PA
			19104-6395
}}
	\date{\vspace{-5ex}}
\maketitle

\begin{abstract}We consider critical percolation on Galton-Watson trees and prove
	quenched analogues of classical theorems of critical branching
	processes. We
	show that the probability critical percolation reaches depth $n$ is
	asymptotic to a tree-dependent constant times $n^{-1}$. Similarly,
	conditioned on critical percolation reaching depth $n$, the number of
	vertices at depth $n$ in the critical percolation cluster almost surely
	converges in distribution to an exponential random variable with mean
	depending only on the offspring distribution. The incipient infinite cluster
	(IIC) is constructed for a.e. Galton-Watson tree and we prove a limit
	law for
	the number of vertices in the IIC at depth $n$, again depending only on the
	offspring distribution. Provided the offspring distribution used to generate
	these Galton-Watson trees has all finite moments, each of these results holds
	almost-surely. 
\end{abstract}

\section{Introduction} \label{sec:intro}

We consider percolation on a locally finite rooted tree $T$: each edge
is open with probability $p \in(0,1)$, independently of all others.
Let $\rtt$ denote the root of $T$ and $\CC_p$ be the open
$p$-percolation cluster of the root. We may consider the \emph
{survival probability} $\theta_T(p) := \P[|\CC_p| = +\infty]$ and
note that $\theta_T$ is an increasing function of $p$. There thus
exists a \emph{critical percolation parameter} $p_c \in[0,1]$ so that
$\theta_T(p) = 0$ for all $p \in[0,p_c)$ and $\theta_T(p) > 0$ for
$p \in(p_c,1]$. If $T$ is a regular tree where each non-root vertex
has degree $d + 1$---i.e. each vertex has $d$ children---then the
classical theory of branching processes shows that $p_c = \frac{1}{d}$
and $\theta_T(p_c) = 0$ (see, for instance, \cite{athreya-ney}).
Since critical percolation does not occur, we may consider the \emph
{incipient infinite cluster} (IIC), in which we condition on critical
percolation reaching depth $M$ of $T$ and take $M$ to infinity.

The IIC for regular trees was first constructed and considered by
Kesten in \cite{kesten-subdiffusive}. In that work, along with \cite
{barlow-kumagai}, the primary focus was on simple random walk on the
IIC for regular trees. Our focus is on three elementary quantities for
random $T$: the probability that critical percolation reaches depth
$n$; the number of vertices of $\mathcal{C}_p$ at depth $n$
conditioned on percolation reaching depth $n$; and the number of
vertices in the IIC at depth $n$. For regular trees, these questions
were answered in the study of critical branching processes. In fact,
these classical results apply to \emph{annealed} critical percolation
on Galton-Watson trees. If we generate a Galton-Watson tree $T$ with
progeny distribution $Z \geq1$ with $\E[Z] > 1$, we may perform $p_c
= 1/\E[Z]$ percolation at the same time as we generate $T$; this is
known at the \emph{annealed} process---in which we generate $T$ and
percolate simultaneously---and is equivalent to generating a
Galton-Watson tree with offspring distribution $\Zt:= \Bin(Z,p_c)$.
Since $\E[\Zt] = 1$, this is a critical branching process and thus
the classical theory can be used:
\begin{thm}[\cite{kesten-ney-spitzer}] \label{thm:annealed}
	Suppose $\E[Z^2] < \infty$, and set $Y_n$ to be the set of vertices
	at depth $n$ of $T$ connected to the root in $p_c = 1/\E[Z]$
	percolation. Then
	\begin{enumerate}
		\item[$(a)$] The annealed probability of surviving to depth $n$
		satisfies
		\[
		n \cdot\P[|Y_n| > 0] \to\frac{2}{\Var[\widetilde{Z}]} = \frac{2
			\E[Z]^2}{\E[Z(Z-1)]}\,.
		\]
		\item[$(b)$] The annealed conditional distribution of $|Y_n|/n$ given
		$|Y_n| > 0$ converges in distribution to an exponential law with mean
		$\frac{\E[Z(Z-1)]}{2\E[Z]^2}$ as $n\to\infty$.
	\end{enumerate}
\end{thm}

Under the additional assumption of $\E[Z^3] < \infty$, parts $(a)$
and $(b)$ are due to Kolmogorov \cite{kolmorogov1938} and Yaglom \cite
{yaglom} respectively; as such, they are commonly referred to as
Kolmogorov's estimate and Yaglom's limit law. For a modern treatment of
these classical results, see \cite{LPP-95} or \cite[Section
$12.4$]{LP-book}. Although less widely known, Theorem \ref
{thm:annealed} quickly gives a limit law for the size of the annealed IIC.

\begin{corollary} \label{cor:IIC-annealed}
	If $\E[Z^2] < \infty$, let $C_n$ denote the number of vertices at
	depth $n$ in the annealed incipient infinite cluster. Then $C_n / n$
	converges in distribution to the random variable with density $\lambda
	^2 x e^{-\lambda x }$ with $\lambda:=\frac{2 \E[Z]^2}{\E[Z(Z-1)]}$
	on $[0,\infty)$. In other words,
	\[
	\lim_{n \to\infty}\left(\lim_{M\to\infty}\P[ |Y_n|/n \in(a,b)
	\| |Y_M| > 0 ] \right)= \int_{a}^b \lambda^2x e^{-\lambda x}\,dx
	\]
	for each $a < b.$
\end{corollary}

This can be easily proven from Theorem \ref{thm:annealed} using an
argument similar to the proof of Theorem \ref{thm:IIC}, and thus the
details are omitted.

Our goal is to upgrade Theorem \ref{thm:annealed} and Corollary \ref
{cor:IIC-annealed} to hold for the \emph{quenched} process; that is,
rather than generate $T$ and perform percolation at the same time as in
the annealed case, we generate $T$ and then perform percolation on each
resulting $T$. Before stating the quenched results, we recall some
notation and facts from the theory of branching processes. If we allow
$\P[Z = 0] > 0$ and condition on the resulting tree being infinite, we
may pass to the reduced tree as in \cite[Chapter $5.7$]{LP-book} in
which we remove all vertices that have finitely many descendants; this
results in a new Galton-Watson process with some offspring distribution
$\tilde{Z} \geq1$. We therefore assume without loss of generality
that $Z \geq1$. For a Galton-Watson tree $T$, let $Z_n$ denote the
number of vertices at distance of $n$ from the root; then the process
$W_n = Z_n / (\E[Z])^n$ converges almost-surely to some random
variable $W$.

A first quenched result is that of \cite{lyons90}, which states that
for a.e.\ supercritical Galton-Watson tree with progeny distribution
$Z$, we have that the critical percolation probability is $p_c = 1/\E
[Z]$; furthermore, for almost every Galton-Watson tree $\TT$, $\theta
_\TT(p) = 0$ for $p \in[0,p_c]$ and $\theta_\TT(p)>0$ for $p \in
(p_c,1]$. For a fixed tree $T$, let $\P_T[\cdot]$ be the probability
measure induced by performing $p_c$ percolation on $T$. When $T$ is
random, this is a random variable and we may ask about the almost sure
behavior of certain probabilities. Our main results are summarized in
the following theorem:

\begin{thm}\label{thm:main}
	Let $\TT$ be a Galton-Watson tree with progeny distribution $Z \geq1$
	with $\E[Z] > 1$. Suppose $\E[Z^p] < \infty$ for each $p \geq1$.
	Set $\lambda:= \frac{2 \E[Z]^2}{\E[Z(Z-1)]}$ and let $Y_n$ be the
	set of vertices in depth $n$ of $\TT$ connected to the root in $p_c =
	1/\E[Z]$ percolation. Then for a.e. $\TT$ we have
	\begin{enumerate}
		\item[$(a)$] $n\cdot\P_\TT[|Y_n| > 0] \to W \lambda$ a.s.
		\item[$(b)$] The conditioned variable $(|Y_n| / n \| |Y_n| > 0)$
		converges in distribution to an exponential random variable with mean
		$\lambda^{-1}$ a.s.
		\item[$(c)$] Let $\mathbf{C}_n$ denote the number of vertices in the
		quenched IIC of $\TT$ at depth $n$. Then $\mathbf{C}_n / n$ converges
		in distribution to the random variable with density $\lambda^2 x
		e^{-\lambda x}$ a.s.
	\end{enumerate}
\end{thm}

Note that, surprisingly, the limit laws of parts $(b)$ and $(c)$ of
Theorem \ref{thm:main} do not depend at all on $\TT$ itself but just
on the distribution of $Z$. This is in sharp contrast to the case of
near-critical and supercritical percolation on Galton-Watson trees, in
which the behavior is dependent on the tree itself \cite
{mpr-quenched}. One possible justification for this lack of dependence
on $W$, for instance, is that conditioning on $|Y_n| > 0$ forces
certain structure of the percolation cluster near the root; since $W$
is mostly determined by the levels of $\TT$ near the root, the
behavior when conditioned on $|Y_n| > 0$ for large $n$ does not depend
on $W$. Part $(a)$ of Proposition \ref{pr:spread} corroborates this
heuristic explanation.

The three parts of Theorem \ref{thm:main} are Theorems \ref
{thm:surv-prob}, \ref{thm:cond-surv} and \ref{thm:IIC} respectively.
The proof of part $(a)$ utilizes its annealed analogue, Theorem \ref
{thm:annealed}$(a)$, along with a law of large numbers argument. Part
$(b)$ is proven by the method of moments building on the work of \cite
{mpr-quenched}. Part $(c)$ follows from there with a similar law of
large numbers argument combined with two short facts about the
structure of the percolation cluster conditioned on $|Y_n| > 0$ (this
is Proposition \ref{pr:spread}).

\begin{remark} Theorem \ref{thm:main} assumes that $\E[Z^p] < \infty
	$ for each $p \geq1$, and we suspect that this condition is an
	artifact of the proof. Since we use the method of moments, it is
	natural that we require all moments of the underlying distribution to
	be finite. We suspect that less rigid conditions are sufficient, but
	this would require a different proof strategy than the method of
	moments, perhaps utilizing a stronger anti-concentration statement in
	the vein of Proposition \ref{pr:spread}.
\end{remark}

\section{Set-up and notation}

We begin with some notation and a brief description of the probability
space on which we will work. Let $Z$ be a random variable taking values
in $\{1,2,\ldots,\}$ with $\mu:= \E[Z] > 1$ and $\P[Z = 0] = 0$.
Define its \emph{probability generating function} to be $\phi(z):=
\sum\P[Z = k] z^k$. Let $\TT$ be a random locally finite rooted tree
with law equal to that of a Galton-Watson tree with progeny
distribution $Z$ and let $(\Omega_1,\T,\GW)$ be the probability
space on which it is defined. Since we will perform percolation on
these trees, we also use variables $\{U_i\}_{i = 1}^\infty$ where the
$U_i$ are i.i.d. random variables uniform on $[0,1]$; let $(\Omega
_2,\F_2, \P_2)$ be the corresponding probability space. Our canonical
probability space will be $(\Omega,\F,\P)$ with $\Omega:= \Omega_1
\times\Omega_2$, $\F:= \T\otimes\F_2$ and $\P:= \GW\times\P
_2$. We interpret an element $\omega=(T,\omega_2) \in\Omega$ as the
tree $T$ with edge weights given by the $U_i$ random variables. To
obtain $p$ percolation, we restrict to the subtree of edges with weight
at most $p$. Since we are concerned with quenched probabilities, we
define the measure $\P_\TT[\cdot] := \P[\cdot\| \TT] = \P[\cdot
\| \T]$. Since this is a random variable, our goal is to prove
theorems $\GW$-a.s.

We employ the usual notation for a rooted tree $T$, Galton-Watson or
otherwise: $\rtt$ denotes the root; $T_n$ is the set of vertices at
depth $n$; and $Z_n := |T_n|$. In the case of a Galton-Watson tree $\TT
$, we define $W_n := Z_n / \mu^n$ and recall that $W_n \to W$ almost
surely. Furthermore, if $\E[Z^p] < \infty$ for some $p \in[1,\infty
)$, we in fact have $W_n \to W$ in $L^p$ \cite[Theorems $0$ and
$5$]{bingham-doney74}. In the Galton-Watson case, define $\T_n :=
\sigma(\TT_n)$; then $(\T_n)_{n = 0}^\infty$ is a filtration that
increases to $\T$. For a vertex $v$ of $T$, define $T(v)$ to be the
descendant tree of $v$ and extend our notation to include $T_n(v),
Z_n(v), W_n(v)$ and $W(v)$. For vertices $v$ and $w$, write $v \leq w$
if $v$ is an ancestor of $w$.

For percolation, recall that the critical percolation probability for
$\GW$-a.e. $\TT$ is $p_c:= 1/\mu$ and that percolation does not
occur at criticality \cite{lyons90}. For vertices $v$ and $w$ with $v
\leq w$, let $\{v \conn w\}$ denote the event that there is an open
path from $v$ to $w$ in $p_c$ percolation; let $\{v \conn(u,w)\}$ be
the event that $v$ is connected to both $u$ and $w$ in $p_c$
percolation; for a subset $S$ of $\TT$, let $\{v \conn S \}$ denote
the event that $v$ is connected to some element of $S$ in $p_c$
percolation; lastly, let $Y_n$ be the set of vertices in $\TT_n$ that
are connected to $\rtt$ in $p_c$ percolation.

\section{Quenched results}
\subsection{Moments}

For $k \geq j$, let $\mathcal{C}_j(k)$ denote the set of
$j$-compositions of $k$, i.e.\ ordered $j$-tuples of positive integers
that sum to $k$. Define
\[
c_{k,j} := p_c^k \sum_{a \in\mathcal{C}_j(k)} m_{a_1} m_{a_2}\cdots m_{a_j}
\]
where $m_r:= \E[\binom{Z}{r}]$. We use the following result from
\cite{mpr-quenched}:

\begin{thm}[\cite{mpr-quenched}] \label{thm:martingales}
	Define
	\[
	M_n^{(k)} := \E_\TT\left[\binom{|Y_n|}{k}\right] - \sum_{i =
		1}^{k-1}c_{k,i} \sum_{j = 0}^{n-1} \E_\TT\left[ \binom{|Y_j|}{i}
	\right] \,.
	\]
	If $\E[Z^{2k}]< \infty$, then $M_n^{(k)}$ is a martingale with
	respect to the filtration $(\T_n)$, and there exist constants $C_k$
	and $c_k$ so that
	\[
	\Vert M_{n+1}^{(k)} - M_n^{(k)} \Vert_{L^2} \leq C_k e^{-c_k n}\,.
	\]
\end{thm}
While Theorem \ref{thm:martingales} is not stated precisely this way
in \cite{mpr-quenched}, the martingale property follows from \cite
[Lemma $4.1$]{mpr-quenched}, while the $L^2$ bound on the increments is
given in \cite[Theorem $4.4$]{mpr-quenched}. This gives us the leading
term of each $\E_\TT\left[|Y_n|^k\right]$.
\begin{proposition}\label{pr:factorial-moments}
	For each $k$,
	\[
	\E_\TT\left[|Y_n|^k\right]n^{-(k-1)} \to k! \left(\frac{p_c^2
		\phi''(1)}{2} \right)^{k-1} W
	\]
	almost surely and in $L^2$.
\end{proposition}
\begin{proof}
	By Theorem \ref{thm:martingales}, $M_n^{(k)}$ is a martingale with
	uniformly bounded $L^2$ norm for each $k$. By the $L^p$ martingale
	convergence theorem, $M_n^{(k)}$ converges in $L^{2}$ and almost
	surely. We now proceed by induction on $k$. For $k = 1$, $\E_\TT
	[|Y_n|] = W_n$ which converges to $W$. Suppose that the proposition
	holds for all $j < k$. Then by convergence of $M_n^{(k)}$,
	\[
	\E_\TT\left[\binom{|Y_n|}{k} \right]n^{-(k-1)} = \sum_{i
		=1}^{k-1} c_{k,i} n^{-(k-1)}\sum_{j = 0}^{n-1} \E_\TT\left[\binom
	{|Y_j|}{i} \right] + o(1)
	\]
	
	where the $o(1)$ term is both in $L^2$ and almost surely.
	By induction, the leading term is the contribution from $i = k-1$.
	Noting that $c_{k,k-1} = (k-1)p_c^2 \frac{\phi''(1)}{2}$ and the fact
	that $\sum_{j = 0}^{n-1} j^d \sim\frac{1}{d+1}n^{d+1}$ completes the proof.
\end{proof}

\subsection{Survival probabilities}

Throughout, define $\lambda:= \frac{2}{p_c^2 \phi''(1)}$. Our first
task is to find a quenched analogue of Kolmogorov's estimate:
\begin{thm} \label{thm:surv-prob}
	If $\E[Z^4] < \infty$, then
	\[
	n\cdot\P_\TT[|Y_n| > 0] \to W \lambda
	\]
	almost surely.
\end{thm}
The proof utilizes the Bonferroni inequalities. In order to control the
second-order term, the variance of a sum of pairs is calculated,
thereby introducing the requirement of $\E[Z^4] < \infty$. We begin
first by proving upper and lower bounds:
\begin{lemma} \label{lem:prop-sandwich}
	For each $n$,
	\begin{equation*}
	\frac{n\cdot\E_\TT[|Y_n|]^2}{\E_\TT[|Y_n|^2]}\leq n\cdot\P_\TT
	[|Y_n| > 0 ] \leq\frac{2 \overline{W}}{1 - p_c}
	\end{equation*}
	where, $\overline{W} = \sup_n W_n$.
\end{lemma}
\begin{proof}
	The lower bound is the Paley-Zygmund inequality. For the upper bound,
	we use \cite[Theorem 5.24]{LP-book}:
	\[
	\P_\TT[|Y_n| > 0] \leq\frac{2}{\mathscr{R}(\rtt\conn\TT_n)}
	\]
	where $\mathscr{R}(\rtt\conn\TT_n)$ is the equivalent resistance
	between the root and $\TT_n$ when all of $\TT_n$ is shorted to a
	single vertex and each edge branching from depth $k-1$ to $k$ has
	resistance $\frac{1 - p_c}{p_c^k}$. Shorting together all vertices at
	depth $k$ for each $k$ gives the lower bound
	\begin{equation*}
	\mathscr{R}(\rtt\conn\TT_n) \geq\sum_{k = 1}^{n}\frac{1 -
		p_c}{Z_k p_c^k} = \sum_{k = 1}^{n}\frac{1 - p_c}{W_k}\geq(1 -
	p_c)\frac{n}{\overline{W}}\,.\qedhere
	\end{equation*}
\end{proof}

\noindent{\emph{Proof of Theorem }\ref{thm:surv-prob}: } For each
fixed $m < n$, the Bonferroni inequalities imply
\begin{equation} \label{eq:bonf}
\left|n\P_\TT[\rtt\conn\TT_n] - n\sum_{v \in\TT_m}\P_\TT
[\rtt\conn v \conn\TT_n]\right| \leq n\sum_{u,v \in\binom{\TT
		_m}{2}} \P_\TT[\rtt\conn(u,v) \conn\TT_n]\,.
\end{equation}

If we can show that the right-hand side of \eqref{eq:bonf} converges
a.s. to zero for some choice of $m = m(n)$, then the survival
probability is sufficiently close to a sum of i.i.d. random variables.
The random variables $\P_\TT[\rtt\conn v \conn\TT_n]$ are i.i.d.\
with mean $p_c^m \P[\rtt\conn\TT_{n - m}]$, implying that the sum
is close to $W_m \P[\rtt\conn\TT_{n - m}]$. Applying the annealed
result Theorem \ref{thm:annealed} would then complete the proof after
noting that $W_m \to W$ almost surely provided $m \to\infty$. The
remainder of the proof follows this sketch.

Set $m = \lceil n^{1/4} \rceil$; we then bound the second moment
\begin{align*}
\E&\left[ \left(\sum_{u,v \in\binom{\TT_m}{2}} \P_\TT[\rtt
\conn(u,v) \conn\TT_n]\right)^2\right] \\
&\qquad=
\E\left[\left(\sum_{u,v \in\binom{\TT_m}{2}}\P_\TT[\rtt\conn
(u,v)]\P_\TT[u\conn\TT_n]\P_\TT[v \conn\TT_n] \right)^2 \right]
\\
&\qquad= \E\left[ \E\left[\left(\sum_{u,v \in\binom{\TT
		_m}{2}}\P_\TT[\rtt\conn(u,v)]\P_\TT[u\conn\TT_n]\P_\TT[v
\conn\TT_n] \right)^2 \, \Bigg| \, \T_m \right]\right] \\
&\qquad= \E\left[ \E\left[\left(\sum_{u,v \in\binom{\TT
		_m}{2}}\P_\TT[\rtt\conn(u,v)]\P_\TT[u\conn\TT_n]\P_\TT[v
\conn\TT_n] \right)^2 \, \Bigg| \, \T_m \right]^{(1/2)\cdot
	2}\right] \\
&\qquad\leq\E\left[ \left(\sum_{u,v \in\binom{\TT_m}{2}} \P
_\TT[\rtt\conn(u,v)] \left\Vert\P_\TT[u \conn\TT_n] \P_\TT[v
\conn\TT_n] \right\Vert_{L^2} \right)^2\right] &&\\
&\qquad\qquad\qquad\qquad\qquad\qquad\qquad\qquad\qquad\qquad\qquad\qquad\qquad\qquad\qquad\text{ by the triangle inequality}& \\
&\qquad\leq\left(\frac{2}{1 - p_c}\right)^4 \E[\overline{W}^2]^2
\cdot(n-m)^{-4} \E\left[\binom{|Y_m|}{2}^2 \right]\qquad\qquad\quad\text{ by
	Lemma }\ref{lem:prop-sandwich}&&\\
&\qquad\leq C m^2 n^{-4}\qquad\qquad\qquad\qquad\qquad\qquad\qquad\qquad\qquad\qquad\qquad\ \  \text{ by Theorem }\ref{thm:martingales}\,.&&
\end{align*}

Multiplying by $n$, the second moment of the right-hand side of \eqref
{eq:bonf} is bounded above by $C m^2 n^{-2} = O(n^{-3/2})$ which is
summable in $n$. By Chebyshev's Inequality together with the
Borel-Cantelli Lemma, the right-hand side of \eqref{eq:bonf} converges
to zero almost surely. This implies
\begin{equation} \label{eq:surv-prob-avg}
n\P_\TT[\rtt\conn\TT_n] = n \sum_{v \in\TT_m} \P_\TT[\rtt
\conn v \conn\TT_n] +o(1) = \sum_{v \in\TT_m} \frac{n \P_\TT[v
	\conn\TT_n]}{\mu^m }+ o(1)\,.
\end{equation}

We want to show that the right-hand side of \eqref{eq:surv-prob-avg}
converges to $W \lambda$, so we first calculate
\begin{align*}
\Var\bigg[ \sum_{v \in\TT_m} &\frac{n \P_\TT[v \conn\TT_n] -
	n \P[\rtt\conn\TT_{n - m}]}{\mu^m} \bigg] \\
&= \E\left[\Var\left[ \sum_{v \in\TT_m} \frac{n \P_\TT[v
	\conn\TT_n] - n \P[\rtt\conn\TT_{n - m}]}{\mu^m} \, \bigg| \,
\T_m\right] \right] \\
&= \E\left[\frac{1}{\mu^{2m}} \sum_{v \in\TT_m} \Var[n \P_\TT
[v \conn\TT_n]] \right] \\
&\leq\frac{C}{\mu^m}
\end{align*}
where the last inequality is via Lemma \ref{lem:prop-sandwich}. Since
this is summable in $n$, Chebyshev's Inequality and the Borel-Cantelli
Lemma again imply
\[
\sum_{v \in\TT_m} \frac{n \P_\TT[v \conn\TT_n]}{\mu^m} = \sum
_{v \in\TT_m} \frac{n \P[\rtt\conn\TT_{n - m}]}{\mu^m} + o(1) =
W_m (n \cdot\P[\rtt\conn\TT_{n - m}]) + o(1)\,.
\]
Taking $n \to\infty$ and utilizing Theorem \ref{thm:annealed}
together with \eqref{eq:surv-prob-avg} completes the proof. $\Cox$

\subsection{Conditioned survival}

\begin{thm} \label{thm:cond-surv}
	Suppose $\E[Z^p] < \infty$ for all $p \geq1$. Then the conditional
	variable $(|Y_n|/n \, | \, |Y_n| > 0)$ converges in distribution to an
	exponential random variable with mean $\lambda^{-1}$ for $\GW$-almost
	every $\TT$.
\end{thm}

By conditional random variable $(|Y_n|/n \, | \, |Y_n| > 0)$, we mean
the random variable with law $\P_\TT[|Y_n|/n \in\cdot\, | \, |Y_n|
> 0 ]$.

\begin{proof} The proof is via the method of moments. In particular,
	since the moment generating function of an exponential random variable
	has a positive radius of convergence, its distribution is uniquely
	determined by its moments. Thus, any sequence of random variables with
	each moment converging to the moment of an exponential random variable
	must converge in distribution to that exponential random variable \cite
	[Theorems $30.1$ and $30.2$]{billingsley}.
	
	Let $X_n$ be a random variable with distribution $(|Y_n|/n \, | \,
	|Y_n| > 0)$. It is sufficient to show $\E_\TT[X_n^k] \to k! \lambda
	^{-k}$ $\GW$-a.s. since $k! \lambda^{-k}$ is the $k$th moment of an
	exponential random variable. Proposition \ref{pr:factorial-moments}
	and Theorem \ref{thm:surv-prob} imply
	\begin{align*}
	\E_\TT[X_n^k] &= \frac{\E_\TT[|Y_n|^k]}{n^k \P_\TT[|Y_n| > 0]} \\
	&= \frac{\E_\TT[|Y_n|^k}{n^{k-1}} \cdot\frac{1}{n\cdot\P_\TT
		[|Y_n| > 0]} \\
	&\to k! W\lambda^{-(k-1)}\cdot\frac{1}{\lambda W} \\
	&= k! \lambda^{-k}\,.\qedhere
	\end{align*}
\end{proof}

More can be said about the structure of the open percolation cluster of
the root conditioned on $\rtt\conn\TT_n$, but we require two
general, more or less standard lemmas first.

\begin{lemma} \label{lem:condition}
	For any events $A$ and $B$ with $\P[B] \neq0$,
	\[
	|\P[A \| B] - \P[A]| \leq\P[B^c]\,.
	\]
\end{lemma}
\begin{proof}
	Expand
	\[
	\P[A] = \P[A \| B](1 - \P[B^c]) + \P[A \| B^c] \P[B^c]
	\]
	and solve
	\[
	\P[A] - \P[A\| B] = (\P[A \| B^c] - \P[A \| B]) \P[B^c]\,.
	\]
	Taking absolute values and bounding $|\P[A \| B^c] - \P[A \| B]| \leq
	1$ completes the proof.
\end{proof}

\begin{lemma} \label{lem:sum-tails}
	Let $X_k$ be i.i.d. centered random variables with $\E[|X_1|^p] <
	\infty$ for some $p \in[2,\infty)$. Then there exists a constant
	$C_p$ so that
	\[
	\P\left[ \left|\sum_{k = 1}^n \frac{X_k}{n}\right| > t \right]
	\leq C_p t^{-p}n^{-p/2} + 2 \exp\left(-\frac{n t^2}{\Var
		[X_1]}\right)
	\]
	for all $t > 0$.
\end{lemma}
\begin{proof}
	This is a straightforward application of \cite[Theorem
	$2.1$]{chesneau} which states that for independent random variables
	$M_i$ with $\E[M_i] = 0$ and $\E[|M_i|^p] < \infty$ for some $p > 2$
	we have
	\[
	\P\left[\left|\sum_{i=1}^n M_i\right| \geq t \right] \leq C_p
	t^{-p}\max\left(r_{n,p}(t), (r_{n,2}(t))^{p/2}\right) + \exp\left
	(- \frac{t^2}{16 b_n}\right)
	\]
	where $r_{n,u}(t) = \sum_{i = 1}^n \E(|M_i|^u \one_{|M_i| \geq
		3b_n/t}),$ $b_n = \sum_{i = 1}^n \E[M_i^2]$ and $C_p$ is a positive
	constant. Setting $M_i = X_i/n$ completes the proof.
\end{proof}

For a fixed tree and $m < n$, define $B_m(n)$ to be the event that
$\rtt\conn\TT_n$ through precisely one vertex at depth $m$.

\begin{proposition} \label{pr:spread}
	Suppose $\E[Z^p] < \infty$ for all $p \geq1$. There exists an $N =
	N(\TT)$ with $N < \infty$ almost surely so that for all $n \geq N$,
	we have
	\begin{enumerate}
		\item[$(a)$] $\P_\TT[B_m(n)^c \| \rtt\conn\TT_n] < Cn^{-1/4}$ for
		$m = m(n) := \lceil\frac{\log n}{4 \log\mu}\rceil$
		\item[$(b)$] $\max_{v \in\TT_n} \P_\TT[v \in Y_n \| \rtt\conn
		\TT_n] = O(n^{-1/8})$
	\end{enumerate}
	for some constant $C > 0$.
\end{proposition}
\emph{Proof.} Note first that for the choice of $m$ as in part $(a)$,
we have $ \frac{1}{2\mu} W n^{1/4} \leq Z_m \leq2 \mu W n^{1/4}$ for
sufficiently large $n$.

$(a)$ Using Theorem \ref{thm:surv-prob} and Lemma \ref
{lem:prop-sandwich}, we bound
\begin{align}
\P_\TT[B_m(n)^c \| \rtt\conn\TT_n] &\leq\frac{\left(\sum_{v
		\in T_m} \P_\TT[v \conn\TT_n] \right)^2}{\P_\TT[\rtt\conn\TT
	_n]} \nonumber\\
&\leq\left(\frac{2}{1 - p_c}\right)^2 \left(\frac{\sum_{v \in
		\TT_m} \overline{W}(v)}{Z_m} \right)^2 \frac{Z_m^2}{(n - m)^2 \P
	_\TT[\rtt\conn\TT_n]} \nonumber\\
&\leq C \left(\frac{\sum_{v \in\TT_m} \overline{W}(v)}{Z_m}
\right)^2 W n^{-1/2} \label{eq:B-n-complement-bound}
\end{align}
for $n$ sufficiently large, and some choice of $C > 0$ depending on the
distribution of $Z$. Applying Lemma \ref{lem:sum-tails} for $p = 9$
gives
\[
\P\left[ \left|\frac{\sum_{v \in\TT_m} \overline{W}(v)}{Z_m} -
\E[\overline{W}] \right| > n^{1/8} \right] \leq C_9 n^{-9/8} + 2
\exp\left(- n^{1/4} / \Var[\overline{W}] \right)
\]
where we use the trivial bound of $1 \leq Z_m.$ Since this is summable
in $n$, the Borel-Cantelli Lemma implies that this event only occurs
finitely often. In particular, this means that for sufficiently large
$n$
\begin{equation} \label{eq:B_m-bound}
\P_\TT[B_m(n)^c \, | \, \rtt\conn\TT_n] \leq C W n^{-1/4}
\end{equation}
for some constant $C > 0$ depending only on the distribution of $Z$.

$(b)$ Applying Lemma \ref{lem:condition} to the measure $\P_\TT[
\cdot\| \rtt\conn\TT_n]$ and recalling $B_m(n) \subseteq\rtt\conn
\TT_n$,
\begin{align*}
\Big|\P_\TT[v \in Y_n \, | \,\rtt\conn\TT_n] - \P_\TT[v \in Y_n
\| B_m(n)] \Big| &\leq\P_\TT\left[B_m(n)^c \| \rtt\conn\TT_n
\right]
\end{align*}
which is $O(n^{-1/4})$ by part $(a)$. It is thus sufficient to bound
$\P_\TT[v \in Y_n \| B_m(n)]$.
For a vertex $v \in\TT_n$ and $m < n$, let $P_m(v)$ be the ancestor
of $v$ in $\TT_m$. We then have
\[
\P_\TT[v \in Y_n \| B_m(n)] \leq\P_\TT[\rtt\conn P_m(v) \conn\TT
_n \| B_m(n)]\,.
\]
Conditioned on $B_m(n)$, there exists a unique vertex $w \in\TT_m$ so
that $\rtt\conn w \conn\TT_n$; this vertex $w$ is chosen with
probability bounded above by
\begin{align}
\P_\TT[\rtt&\conn w \conn\TT_n \| B_m(n)] \nonumber\\
&\leq\frac{\P_\TT[\rtt\conn w \conn\TT_n]}{\sum_{u \in\TT_m}
	\P_\TT[\rtt\conn u \conn\TT_n] - \sum_{(u_1,u_2) \in\binom{\TT
			_m}{2}} \P_\TT[\rtt\conn(u_1,u_2) \conn\TT_n] } \nonumber\\
&\leq\frac{\P_\TT[w \conn\TT_n]}{ \sum_{u \in\TT_m} \P_\TT[u
	\conn\TT_n] - \left(\sum_{u \in\TT_m}\P_\TT[u \conn\TT_n]
	\right)^2 } \nonumber\\
&\leq\frac{c (n - m)^{-1}\overline{W}(w)}{(1 + o(1))\sum_{u \in\TT
		_m} \P_\TT[u \conn\TT_n]} \label{eq:point-B_m-bound}
\end{align}

where the latter inequality is by applying the bound of Lemma \ref
{lem:prop-sandwich} to the numerator and arguing as in \eqref
{eq:B-n-complement-bound} to almost-surely bound the denominator. In
particular, the $o(1)$ term is uniform in $w$.

We want to take the maximum over all possible $w \in\TT_m$, and note
that for any $\alpha> 0$,
\begin{align*}
\P\left[\max_{w \in\TT_m} \overline{W}(w) > n^{\alpha}\right]
&=\E\left[\P\left[\max_{w \in\TT_m} \overline{W}(w) > n^{\alpha
}\, \big|\, \T_m \right]\right] \\
&\leq \E[Z_m] \P[\overline{W}> n^\alpha] \\
&\leq\mu^m \cdot\frac{\E[\overline{W}^{2/\alpha}]}{n^{2}} \\
&= O(n^{-7/4})
\end{align*}
which is summable, implying that for any fixed $\alpha\!>\!0$, we
eventually have $\max_{w \in\TT_m} \overline{W}(w)\!\leq\!n^{\alpha
}.$ It merely remains to bound the denominator of \eqref{eq:point-B_m-bound}.

Note that by Proposition \ref{pr:factorial-moments}, the lower bound
given in Lemma \ref{lem:prop-sandwich} converges almost surely to
$\frac{W \lambda}{2}$ as $n \to\infty$. In particular, this means
that if we set
\[
p_n := \P\left[\frac{W \lambda}{4} \leq n\P_\TT[|Y_n| > 0] \right],
\]
then $p_n \to1$. By Hoeffding's inequality together with
Borel-Cantelli, the number of vertices $u \in\TT_m$ for which we have
\[
\frac{W(u) \lambda}{4} \leq(n - m) \P_\TT[u \conn\TT_n]
\]
is almost surely at least $1/2$ of $\TT_m$ for $n$ sufficiently large.
This gives
\[
(n-m)\sum_{u \in\TT_m} \P_\TT[u \conn\TT_n] \geq\frac{\lambda
}{4}\sum_{u \in\TT_m} W(u) \one_{W(u)\lambda/4 \leq(n - m)\P_\TT
	[u \conn\TT_n] } = \Omega(Z_m) \,.
\]

Recalling that $Z_m = \Theta(Wn^{-1/4})$ and plugging the above into
\eqref{eq:point-B_m-bound} completes the proof. $\Cox$

\subsection{Incipient infinite cluster}

As in \cite{kesten-IIC}, we sketch a proof of the construction of the
IIC. For an infinite tree $T$, define $T{[n]}$ to be the finite subtree
of $T$ obtained by restricting to vertices of depth at most $n$.

\begin{lemma}\label{lem:constr}
	Suppose $\E[Z^4] < \infty$; for a subtree $t$ of $\TT{[n]}$, we have
	\[
	\lim_{M \to\infty} \P_\TT[\mathcal{C}_{p_c}[n] = t \| \rtt\conn
	\TT_M ] = \frac{\sum_{v \in t_n}W(v)}{W}\P_\TT[\mathcal
	{C}_{p_c}[n] = t ]
	\]
	almost surely for each tree $t$.
	
	The random measure $\mu_\TT$ on subtrees of $\TT$ with marginals
	\[
	\mu_\TT\big|_{\T_n}[t] := \frac{\sum_{v \in t_n}W(v)}{W}\P_\TT
	[\mathcal{C}_{p_c}[n] = t ]
	\]
	has a unique extension to a probability measure on rooted infinite
	trees $\GW$ almost surely. The IIC is thus the random subtree of $\TT
	$ with law $\mu_\TT$.
\end{lemma}
\begin{proof}
	Since each $\TT$ has countably many vertices, Theorem \ref
	{thm:surv-prob} assures that $n\P_\TT[v \conn\TT_{n + |v|}] =
	\lambda W(v)$ for each vertex $v$ of $\TT$ a.s. When all of these
	limits hold, we then have
	\begin{align*}
	\P_\TT[\mathcal{C}_{p_c}[n] = t \| \rtt\conn\TT_M] &= \frac{ \P
		_\TT[\mathcal{C}_{p_c}[n] = t, \rtt\conn\TT_M]}{\P_\TT[\rtt
		\conn\TT_M]} \\
	&= \P_\TT[\mathcal{C}_{p_c}[n] = t]\left(\frac{\sum_{v \in t_n}
		\P_\TT[v \conn\TT_M] + O(|t_n|^2 M^{-2})}{\P_\TT[\rtt\conn\TT
		_M]} \right) \\
	&\xrightarrow{M\to\infty} \P_\TT[\mathcal{C}_{p_c}[n] = t]\frac
	{\sum_{v \in t_n} W(v)}{W}
	\end{align*}
	for each $t$. To show that the measure $\mu_\TT$ can be extended, we
	note that its marginals are consistent, as can be seen via the
	recurrence $W(v) = p_c \sum_{w} W(w)$ where the sum is over all
	children of $v$. Applying the Kolmogorov extension theorem \cite
	[Theorem $2.1.14$]{durrett4} completes the proof.
\end{proof}

It is easy to show that the law of the IIC can in fact be generated by
conditioning on $p > p_c$ percolation to survive and then taking $p \to
p_c^+$:
\begin{corollary} For a subtree $t$ of $\TT[n]$, we have
	\[
	\lim_{p \to p_c^+} \P_\TT[\mathcal{C}_p[n] = t \| |\mathcal{C}_p|
	= \infty] = \frac{\sum_{v \in t_n}W(v) }{W} \P_\TT[\mathcal
	{C}_{p_c}[n] = t]
	\]
	almost surely.
\end{corollary}
\begin{proof}
	As shown in \cite{mpr-quenched}, we have
	\[
	\lim_{p \to p_c} \frac{\P_\TT[|\mathcal{C}_p| = \infty]}{p - p_c}
	= K W
	\]
	almost-surely for some constant $K$ depending only on the offspring
	distribution. The Corollary follows from Bayes' theorem in the same
	manner as Lemma \ref{lem:constr}.
\end{proof}

In light of Lemma \ref{lem:constr}, it is natural to guess that the
number of vertices in the IIC at depth $n$ will asymptotically be the
size-biased version of $(|Y_n| \| \rtt\conn\TT_n )$: the sum $\sum
_{v \in t_n} W(v)$ will be relatively close to $|t_n| W$, therefore
biasing each choice of $t$ by a factor of $|t_n|$. In order to make
this argument rigorous, we will invoke Proposition \ref{pr:spread}
which shows that no single vertex has high probably of surviving
conditionally. Throughout, we use the notation $n(a,b) = (na,nb)$ for
$a < b$ and $\mathbf{C}$ to denote the IIC.

\begin{thm} \label{thm:IIC}
	Suppose $\E[Z^p] < \infty$ for each $1 \leq p < \infty$. Then for
	each $0 \leq a < b$,
	\[
	\lim_{n \to\infty} \P_\TT[\mathbf{C}_n \in n(a,b)] = \int_{a}^b
	\lambda^2 x e^{-\lambda x} \,dx
	\]
	almost surely. In fact, $\mathbf{C}_n / n$ converges in distribution
	to the random variable with density $\lambda^2 x e^{-\lambda x}$ for
	$\GW$-almost every $\TT$.
\end{thm}
\begin{proof}
	To see that convergence in distribution follows from the almost sure
	limit, apply the almost sure limit to each interval $(a,b)$ with $a,b
	\in\mathbb{Q}$; since there are only countably many such intervals,
	there exists a set of full $\GW$ measure on which these limits
	simultaneously exist for each rational interval, thereby implying
	convergence in distribution \cite[Theorem 3.2.5]{durrett4}.
	
	We have
	\[
	\P_\TT[\mathbf{C}_n \in n(a,b)] = \lim_{M \to\infty} \P_\TT[Y_n
	\in n(a,b)\| \rtt\conn\TT_{n + M}]\,.
	\]
	
	For a fixed $n$, write
	\begin{align}
	\P_\TT[&|Y_n| \in n(a,b) \| \rtt\conn\TT_{n+M}] \nonumber\\ &=
	\frac{\P_\TT[\rtt\conn\TT_{n + M} \| |Y_n| \in n(a,b)]\cdot\P
		_\TT[|Y_n| \in n(a,b) \| \rtt\conn\TT_n] \cdot\P_\TT[\rtt\conn
		\TT_n]}{\P_\TT[\rtt\conn\TT_{n + M}]}\,. \label{eq:Y_n-size-cond}
	\end{align}
	
	We then calculate
	\begin{align*}
	\P_\TT&[\rtt\conn\TT_{n + M} \| |Y_n| \in n(a,b)]\\
	&= \sum_{S} \P_\TT[Y_n = S \| |Y_n| \in n(a,b) ] \P_\TT[S \conn\TT_{n + M}] \\
	&= \sum_{S} \P_\TT[Y_n = S \| |Y_n| \in n(a,b) ] \sum_{v \in S} \P
	_\TT[v \conn\TT_{n + M}] + O(M^{-2}) \\
	&= \sum_{v \in\TT_n} \P_\TT[v \in Y_n \| |Y_n| \in n(a,b)] \P_\TT
	[v \conn\TT_{n + M}] + O(M^{-2})\,.
	\end{align*}
	
	For a fixed $n$, we take $M \to\infty$ and utilize Theorem \ref
	{thm:surv-prob} to get
	\begin{equation} \label{eq:p-IIC}
	\lim_{M\to\infty} \frac{\P_\TT[\rtt\conn\TT_{n + M} \| |Y_n|
		\in n(a,b)]}{\P_\TT[\rtt\conn\TT_{n + M}]} = \frac{1}{W} \sum_{v
		\in\TT_n} \P_\TT[v \in Y_n \| |Y_n| \in n(a,b)]\cdot W(v)\,.
	\end{equation}
	We plug this into \eqref{eq:Y_n-size-cond} to get the limit
	\begin{align*}
	\lim_{M \to\infty} & \P_\TT[|Y_n| \in n(a,b) \| \rtt\conn\TT
	_{n+M}] \nonumber\\
	&= \left(\sum_{v \in\TT_n} \frac{ \P_\TT[v \in Y_n \| |Y_n| \in
		n(a,b)] }{n} \cdot W(v)\right)\nonumber\\
	&\quad\times \left(\P_\TT[|Y_n| \in n(a,b) \|
	\rtt\conn\TT_n]\right) \left( \frac{n\cdot\P_\TT[\rtt\conn
		\TT_n]}{W} \right)\,.
	\end{align*}
	
	Theorems \ref{thm:surv-prob} and \ref{thm:cond-surv} show that the
	latter two factors above have almost sure limits $\int_a^b \lambda
	e^{-\lambda x} \,dx$ and $\lambda$ as $n \to\infty$, leaving only
	the first term. We note that
	\begin{align*}
	\E\left[\sum_{v \in\TT_n} \frac{ \P_\TT[v \in Y_n \| |Y_n| \in
		n(a,b)] }{n} \cdot W(v) \,\Bigg|\, \T_n \right] &= \sum_{v \in\TT
		_n}\frac{ \P_\TT[v \in Y_n \| |Y_n| \in n(a,b)] }{n} \\
	&= \E_\TT\left[\frac{|Y_n|}{n} \, \bigg| \, |Y_n| \in n(a,b)
	\right] \\
	&= \frac{\E_\TT\left[\frac{|Y_n|}{n}\cdot\one_{|Y_n|/n \in
			(a,b)} \| \rtt\conn\TT_n \right]}{\P_\TT\left[\frac{|Y_n|}{n}
		\in(a,b) \| \rtt\conn\TT_n \right]}\\
	&\to\frac{\int_{a}^{b} \lambda x e^{-\lambda x} \,dx}{\int_{a}^{b}
		\lambda e^{-\lambda x}\,dx}
	\end{align*}
	where the limit is by the continuous mapping theorem \cite[Theorem
	3.2.4]{durrett4} and Theorem \ref{thm:cond-surv}. It's thus sufficient
	to show that
	\begin{equation}\label{eq:need} \left|\sum_{v \in\TT_n} \frac{ \P
		_\TT[v \in Y_n \| |Y_n| \in n(a,b)] }{n} \cdot(W(v) - 1) \right|
	\xrightarrow{n \to\infty} 0
	\end{equation}
	almost surely.
	
	Our strategy is to use a conditional version of the Borel-Cantelli
	Lemma together with Chebyshev's inequality. We bound the conditional
	variance
	\begin{align}
	\Var\Bigg[\sum_{v \in\TT_n}&\frac{ \P_\TT[v \in Y_n \| |Y_n|
		\in n(a,b)] }{n} \cdot(W(v) - 1) \, \bigg| \, \T_n \Bigg] \nonumber
	\\
	&= \Var(W)\sum_{v \in\TT_n} \frac{\P_\TT[v \in Y_n \| |Y_n| \in
		n(a,b)]^2 }{n^2} \nonumber\\
	&\leq\Var(W) \max_{v \in\TT_n} \P_\TT[v \in Y_n \| |Y_n| \in
	n(a,b)] \sum_{v \in\TT_n} \frac{\P_\TT[v \in Y_n \| |Y_n| \in
		n(a,b)] }{n^2} \nonumber\\
	&\leq\Var(W) \max_{v \in\TT_n} \P_\TT[v \in Y_n \| |Y_n| \in
	n(a,b)] \cdot\frac{\E[Y_n \| |Y_n| \in n(a,b)] }{n^2} \nonumber\\
	& \leq\Var(W) \cdot\frac{b}{n} \cdot\max_{v \in\TT_n} \P_\TT
	[v \in Y_n \| |Y_n| \in n(a,b)] \label{eq:var-bound} \,.
	\end{align}
	We want to show that this is summable, and thus look to bound the $\max
	$ term. Applying Lemma \ref{lem:condition} to the measure $\P_\TT
	[\cdot\| |Y_n| \in n(a,b)]$ gives
	\begin{align}
	&\left|\P_\TT[v \in Y_n \| |Y_n| \in n(a,b)] - \P_\TT[v \in Y_n \|
	|Y_n| \in n(a,b), B_m(n)] \right|\nonumber\\
	&\qquad\qquad\qquad\qquad\qquad\qquad\quad\leq\P_\TT[B_m(n)^c \| |Y_m| \in n(a,b)] \nonumber\\
	&\qquad\qquad\qquad\qquad\qquad\qquad\quad\leq\frac{\P_\TT[B_m(n)^c \| \rtt\conn\TT_n]}{\P_\TT[|Y_n|
		\in n(a,b) \| \rtt\conn\TT_n]} \nonumber \\
	&\qquad\qquad\qquad\qquad\qquad\qquad\quad= O(n^{-1/4}) \label{eq:point-difference}
	\end{align}
	by Proposition \ref{pr:spread} and Theorem \ref{thm:cond-surv}.
	Similarly,
	\begin{align}
	\P_\TT[v \in Y_n \| |Y_n| \in n(a,b), B_m(n)] &= \frac{\P_\TT[v
		\in Y_n, |Y_n| \in n(a,b), B_m(n)]}{\P_\TT[|Y_n| \in n(a,b), B_m(n)]}
	\nonumber\\
	&\leq\frac{ \P_\TT[v \in Y_n,B_m(n)]}{\P_\TT[|Y_n| \in n(a,b),
		B_m(n)]} \nonumber\\
	&= \frac{\P_\TT[v \in Y_n \| B_m(n)]}{\P_\TT[|Y_n| \in n(a,b) \|
		B_m(n)]} \label{eq:point-up}\,.
	\end{align}
	
	Using Lemma \ref{lem:condition} once again expands the denominator
	\[
	\Big| \P_\TT[|Y_n| \in n(a,b) \| B_m(n)] - \P_\TT[|Y_n| \in n(a,b)
	\| \rtt\conn\TT_n]\Big| \leq\P_\TT[B_m(n)^c \| \rtt\conn\TT
	_n] \leq C n^{-1/4}
	\]
	by Proposition \ref{pr:spread}. Plugging into \eqref{eq:point-up}
	gives the upper bound
	\begin{equation}
	\P_\TT[v \in Y_n \| |Y_n| \in n(a,b), B_m(n)] \leq\frac{\P_\TT[v
		\in Y_n \| B_m(n)]}{\P_\TT[|Y_n| \in n(a,b) \| \rtt\conn\TT_n] - C
		n^{-1/4}} \label{eq:point-final}\,.
	\end{equation}
	Combining \eqref{eq:point-difference}, \eqref{eq:point-final} and
	Proposition \ref{pr:spread} bounds
	\[
	\max_{v \in\TT_n} \P_\TT[v \in Y_n \| |Y_n| \in n(a,b)] =
	O(n^{-1/8})\,.
	\]
	Thus, by \eqref{eq:var-bound}, the conditional variance is almost
	surely summable. For any fixed $\delta> 0$, Chebyshev's inequality
	then implies
	\[
	\P\left[\left|\sum_{v \in\TT_n} \frac{\P_\TT[v \in Y_n \|
		|Y_n| \in(a,b)]}{n}\cdot(W(v) - 1) \right| > \delta\,\Bigg|\, \T
	_n\right]
	\]
	is summable almost surely. Applying a conditional Borel-Cantelli Lemma
	(e.g. \cite{chen-bc}) shows that \eqref{eq:need} holds almost surely.
\end{proof}

\section*{Acknowledgements}
The author would like to thank Josh Rosenberg for helpful
conversations.

\bibliographystyle{alpha}
\bibliography{Bib}

\end{document}